\documentclass[11pt,a4paper]{article}
\usepackage[utf8]{inputenc}
\usepackage{amssymb}
\usepackage{amsmath}
\usepackage{amscd}
\usepackage{amsthm}
\usepackage{enumerate}
\usepackage{array}
\usepackage{tabularx}
\usepackage{t1enc}
\usepackage{amsthm}
\usepackage[mathscr]{eucal}
\usepackage{graphicx}
\usepackage{xcolor}


\usepackage{subcaption}

\numberwithin{equation}{section}
\newtheorem{theorem}{Theorem}[section]
\newtheorem{prop}[theorem]{Proposition}
\newtheorem{lemma}[theorem]{Lemma}
\newtheorem{defn}[theorem]{Definition}

\newtheorem{cor}[theorem]{Corollary}
\newtheorem{conj}[theorem]{Conjecture}

\newtheorem{problem}[theorem]{Problem}

\newtheorem{nota}[theorem]{Notation}
\newtheorem{construction}[theorem]{Construction}
\newtheorem{remark}[theorem]{Remark}
\newtheorem{obs}[theorem]{Observation}
\newtheorem{const}[theorem]{Construction}

\newcommand{\cl}{\mathrm{cl}}
\newcommand{\Val}{\mathrm{Val}}

\newcommand{\STS}{\mathrm{STS}}
\newcommand{\Supp}{\mathrm{Supp}}

\newcommand{\cut}[1]{}

\begin{document}

\title{Spreading linear triple systems and expander triple systems}

\author{{\bf Zolt\'an L. Bl\'azsik\thanks{The research was supported by the Hungarian National Research, Development and Innovation Office, OTKA grant no. SNN 132625.},  Zolt\'an L\'or\'ant Nagy\thanks{The author is supported by the Hungarian Research Grants (NKFI) No. K 120154 and SNN 132625 and by the János Bolyai Scholarship of the Hungarian Academy of Sciences}}\\
{\small MTA--ELTE Geometric and Algebraic Combinatorics Research Group}\\
{\small  E\"otv\"os Lor\'and University, Budapest, Hungary}\\
{\small Department of Computer Science}\\
{\small H--1117 Budapest, P\'azm\'any P.\ s\'et\'any 1/C, Hungary}\\
{\small \tt{blazsik@caesar.elte.hu}, \tt{ nagyzoli@cs.elte.hu}}}

\date{}

\maketitle

\begin{abstract}
The existence of Steiner triple systems $\STS(n)$ of order $n$ containing no nontrivial subsystem is well known for every admissible $n$. We generalize this result in two ways. First we define the expander property of $3$-uniform hypergraphs and show the existence of Steiner triple systems which are almost perfect expanders.

Next  we define the strong and weak spreading property of linear hypergraphs, and determine the minimum  size of a linear triple system with these properties, up to a small constant factor.
This property is strongly connected to the connectivity of the structure and of the so-called influence maximization.
We also discuss how the results are related to Erdős' conjecture on locally sparse $\STS$s, influence maximization, subsquare-free Latin squares and possible applications in finite geometry.\\

\end{abstract}


\section{Introduction}
\label{S:1}

A Steiner triple system $\mathcal{S}$ of order $n$, briefly $\STS(n)$, consists of an $n$-element set $V$ and a
collection of triples (or blocks) of $V$, such that every pair of distinct points in
$V$ is contained in a unique block. It is well known due to Kirkman \cite{Kirk} that there exists an $\STS(n)$ if and only if
$n \equiv 1, 3 \pmod 6$, these values are called \emph{admissible}.
 Steiner triple systems correspond to triangle decompositions of the complete graph $G=K_n$. In the context of triangle decompositions of a graph $G$, an edge will always refer to a pair of vertices which is contained in one triple of a certain triple system, $E(G)$ denotes the edge set of $G$, while $|\mathcal{S}|$ is the number of triples in the system, which obviously equals $\frac{1}{3}|E(G)|$ in the case of triple systems obtained from triangle decompositions of a graph $G$.

A triple system induced by a proper subset $V' \subset V$ consists of those triples whose elements do not contain any element of $V \setminus V'$. A nontrivial Steiner subsystem of $\mathcal{S}$ is an $\STS(n')$ induced by a proper subset $V'\subset V$, with $|V'|=n'>3$. Speaking about any triple system's subsystem, we always suppose that it is of order greater than $3$. Similarly but not analogously, we  call a subset $V'\subset V$ of the underlying set of a triple system $\mathcal{F}$ \emph{nontrivial} if it has size at least $3$ and it is not an element of the triple system. Our aim is to generalize and strengthen the results concerning the subsystem-free property of Steiner triple systems, and in general, linear triple systems.  This general concept in incidence geometry is also studied and they are called partial linear spaces.
This paper is devoted to the study of two main features of linear triples systems in an extremal hypergraph theory aspect. The first property is the \emph{expander property} while the second is the so-called \emph{spreading property}.

In 1973, Erdős formulated the following conjecture.

\begin{conj}[Erdős, \cite{Erdos}]
For every $k\geq 2$ there exists a threshold $n_k$ such that for all admissible
$n> n_k$, there exists a Steiner triple system of order $n$ with the following property: every $t+2$ vertices induce less than $t$ triples of $\mathcal{S}$ for $2\leq t\leq k$.
\end{conj}

This conjecture is still open, although recently Glock,  Kühn,  Lo and Osthus  \cite{Kuhn} and independently Bohman and Warnke \cite{bohman} proved its asymptotic version. In other words, this conjecture asserts the existence of arbitrarily sparse Steiner triple systems.\\
One should note here that it is also a natural question whether typical Steiner triple systems are sparse in a very robust sense, namely that they do not contain Steiner subsystems. Indeed, this is equivalent to avoiding a set of $t<n$ vertices inducing quadratically many, $\frac{1}{3}\binom{t}{2}$ triples. The first result in this direction was due to Doyen \cite{Doyen}, who proved the existence of {\em at least one}  subsystem-free $\STS(n)$ for every admissible order $n$.
In the language of triangle decompositions of the edge set,  a subsystem-free $\STS$ may be seen as a decomposition where every subset $V'\subset V(G)$ contains at least one edge which belongs to a triangle not induced by $V'$. 
In order to capture this phenomenon and its generalisation, we require some notation and definitions.

\begin{defn} Given a $3$-uniform linear hypergraph  $\mathcal{F}$ (i.e. linear triple system),  let $E(\mathcal{F})$ be the collection of  vertex pairs $(x, y)$ for which
there exists a triple $(x, y, z)$ from  the system $\mathcal{F}$, containing $x$  and $y$. The corresponding graph $G(\mathcal{F})$ is referred to as
the \emph{shadow  or skeleton of the system}.
\end{defn}

\begin{defn}[Closure and spreading property]\label{szomszed} Consider a graph $G=G(V,E)$ that admits a triangle decomposition. This decomposition corresponds to a  linear triple system $\mathcal{F}$. For an arbitrary set $V' \subset V$, $N(V')$ denotes the set of  its \emph{neighbours}: $$z\in N(V') \Leftrightarrow z\in V \setminus V' \mbox{ \ and \ } \exists xy \in E(G[V']): \{x,y,z\}\in \mathcal{F}.$$
The \emph{closure} $\cl(V')$ of a subset $V'$ w.r.t. a (linear) triple system $\mathcal{F}$ is the smallest set $W \supseteq V'$ for which $|N(W)|=0$ holds. Note that the closure uniquely exists for each set $V'$. We call a (linear) triple system $\mathcal{F}$ \emph{spreading} if $\cl(V')=V$ for every nontrivial subset $V'\subset V$.
\end{defn}

Consequently,   a $\STS(n)$  is subsystem-free if and only if $|N(V')|>0$ holds for all nontrivial subsets $V'$ of  the underlying set  $V$ of the system. Note that  Doyen  used the term  \textit{non-degenerate plane} for $\STS$s with the spreading property  \cite{Doyen, Doyen2}.

Two natural extremal questions arise here. The first one concerns the lower bound on  $|N(V')|$ in terms of $|V'|$ in the case of Steiner triple systems, while the second one seeks for edge-density conditions on triangle decompositions of general graphs $G=G(V,E)$, i.e. linear triple systems, where the condition  $|N(V')|>0$ must hold for all nontrivial subsets of $V$.

\begin{problem}[Expander STSs] Let us call a Steiner triple system $\varepsilon$-expander if there exists some $\varepsilon>0$ such that for every nontrivial $V'\subset V(G)$, $\frac{|N(V')|}{|V'|}\geq \varepsilon$ holds provided that $|V'|\leq \frac{|V|}{2}$. Does there exist an infinite family of $\varepsilon$-expander Steiner triple systems $\STS(n)$ for some $\varepsilon>0$? How large $\varepsilon>0$ can be? 
\end{problem}

This can be interpreted as the analogue of the expander property of graphs and the vertex isoperimetric number \cite{Alon-Boppana}. Similar generalized concepts for expanding triple systems were introduced very recently by Conlon and his coauthors \cite{Conlon, Conlon2}, see also the related paper \cite{Fox-Pach}. Observe however that their definition is slightly different for a triple system to be expander, as Conlon defines the edge-, resp. triple-neighbourhood of a subset $F$ to be 
the edges of the skeleton not contained in $F$ which determine a triple with a point in $F$, or triples not induced by $F$ which has a point in $F$, respectively.

\begin{problem}[Sparse spreading linear triple systems] What is the minimum size  $\xi_{sp}(n)$ of a linear spreading triple system $\mathcal{F}$ on $n$ vertices?
\end{problem}

For these triple systems, the {closure} of any nontrivial subset with respect to the underlying graph of the triple system is the whole system.

Note that one might require only a weaker condition, namely that the closure of any nontrivial  subset of the triple system $\mathcal{F}$ (i.e. consisting of at least two triples) should be the whole system. In applications this condition is equally important, since it models whether every set of hyperedges has a direct influence on the whole system.  For this concept, we introduce the following notation.

\begin{nota} A triple system  $\mathcal{F}$  is \emph{ weakly spreading} if
 $\cl(V')=V$ holds for  every $$V'=V\left( \mathcal{F}' \right) \  : \ \mathcal{F}'\subseteq \mathcal{F}, \   |\mathcal{F}'|>1.$$

\end{nota}

It is easy to see that one gets the same concept if the condition $\cl(V')=V$ is required for every  $V'=V\left( \mathcal{F}' \right)$ for which $\exists \mathcal{F}'\subseteq \mathcal{F}, \   |\mathcal{F}'|=2$. On the other hand, the condition $\cl(V')=V$ is essential in the sense that $\cl(V')\supset V$ is a much weaker property, see Construction \ref{teljes}.

\begin{problem}[Sparse weakly spreading linear triple systems] What is the minimum size $\xi_{wsp}(n)$ of a linear weakly spreading triple system $\mathcal{F}$ on $n$ vertices?
\end{problem}

Our main results are as follows.

\begin{theorem}\label{expander}
For odd prime number $p$, there exists a Steiner triple system $\STS(3p)$ of order ${3p}$, for which $${|N(V')|}\geq {|V'|}-3$$ for every  $V'\subset V(G)$ of size  $|V'|\leq \frac{|V|}{2}$.
\end{theorem}

The result is clearly sharp since $V'$ can be chosen to be the elements of a triple.

\begin{cor}\label{corall}
For every sufficiently large $n$, there exists a Steiner triple system $\STS(\overline{n})$ of order $\overline{n}$, for which $${|N(V')|}\geq {|V'|}-3$$ for every  $V'\subset V(G)$ of size  $|V'|\leq \frac{|V|}{2}$, where $\overline{n}\in [n-n^{0.525},n]$ due to the result of Baker, Harman and Pintz on the difference between consecutive primes \cite{Pintz}. Consequently, for every $n$ one can find a Steiner triple system $\mathcal{S}$ of size $|\mathcal{S}|=(1+o(1))\frac{n^2}{6}$ which is almost $1$-expander. Indeed, we may simply take a construction via Theorem \ref{expander} for a prime $p\in [n-n^{0.525},n]$. 
\end{cor}

As we will see, much smaller edge density compared to that of $\STS$s' still enables us to construct spreading linear triple systems.

\begin{theorem}\label{spreading}
For the minimum size of a spreading linear triple system, we have
 $$ 0.1103n^2<\xi_{sp}(n)$$ while  $$\xi_{sp}(n)<\left(\frac{5}{36}+o(1)\right)n^2\approx 0.139n^2$$ holds for infinitely many $n$.
\end{theorem}

Surprisingly, the weak spreading property does not require a dense structure at all.

\begin{theorem}\label{weaklysp}
For the minimum size of a weakly  spreading linear triple system, we have
$$n-3\leq \xi_{wsp}(n)< \frac{8}{3}n + o({n}).$$
\end{theorem}

The paper is organised as follows. Section 2 is devoted to the expander property of Steiner triple systems and related questions. In Section 3, we make a connection between $k$-connectivity of $3$-graphs and the spreading property, and prove Theorem \ref{spreading} and Theorem \ref{weaklysp}.  Finally, in Section 4 we discuss related problems  concerning Latin squares and influence maximization, and possible applications, notably in the field of finite geometry.

\section{Expander property of Steiner triple systems}

In order to prove Theorem \ref{expander}, we recall first the $\STS$ construction of Bose and Skolem for $n=6k+3$ where $2k+1$ is a prime number, and the well-known Cauchy-Davenport theorem with its closely related variant, the result of Dias da Silva and Hamidoune  about the conjecture of Erdős and Heilbronn. We refer to the book of Tao and Vu \cite{additive} on the subject.

\begin{theorem}[Cauchy-Davenport]\label{CD}
For any prime $p$ and nonempty subsets $A$ and $B$ of the prime order cyclic group $\mathbb{Z}_p$, the size of the sumset $A+B=\{a_i+b_j\  | \ a_i\in A, b_j \in B\}$ can be bounded as
$|A+B|\ge\min\{p,\ |A|+|B|-1\}$.
\end{theorem}

\begin{theorem}[Erdős-Heilbronn conjecture, Dias da Silva and Hamidoune '94]\label{EHcon}
For any prime $p$ and any subset $A$ of the prime order cyclic group $\mathbb{Z}_p$, the size of the restricted sumset $A\dot{+}A=\{a_i+a_j \ | \ a_i\neq a_j\in A\}$ can be bounded as
$|A\dot{+}A|\ge\min\{p,\ 2|A|-3\}$.
\end{theorem}

\begin{const}[Bose and Skolem, case  $n=6k+3$]\label{Bose}
Let the triple system $\mathcal{S}$ be defined in the following way. The underlying set is partitioned into three sets of equal sizes, $V(\mathcal{S})=A\cup B\cup C$, where $|A|=|B|=|C|=2k+1$. Elements of each partition class are indexed by the elements of the additive group $\mathbb{Z}_{2k+1}$. The system $\mathcal{S}$ contains the triple $T$, if
\begin{itemize}
    \item $T=\{a_i, b_i, c_i\}$, $i \in  \mathbb{Z}_{2k+1}$, or
    \item $T=\{a_i, a_j, b_k\}$, $i\neq j \in  \mathbb{Z}_{2k+1}, \  k=\frac{1}{2}(i+j)$, or
    \item $T=\{b_i, b_j, c_k\}$, $i\neq j \in  \mathbb{Z}_{2k+1}, \ k=\frac{1}{2}(i+j)$, or
    \item $T=\{c_i, c_j, a_k\}$, $i\neq j \in  \mathbb{Z}_{2k+1}, \ k=\frac{1}{2}(i+j)$.
\end{itemize}
\end{const}

See \cite{Stinson} for further details and generalisations.

\begin{proof}[Proof of Theorem \ref{expander}]
Let us apply Construction \ref{Bose} with $p=2k+1$ prime. Consider a subset $V_0=A_0\cup B_0\cup C_0$ of the underlying set $V(\mathcal{S})=A\cup B\cup C$, where $|V_0|\leq \frac{n}{2}$. In order to bound $N(V_0)$, we prove a lower bound on $V_0 \cup N(V_0)$. Observe that a vertex $v$ belongs to $N(V_0)$ if and only if there exist two elements of $V_0$ together with which it forms a triple of $\mathcal{S}$. Let us denote by $A^*, B^*, C^*$ the restrictions of $V_0 \cup N(V_0)$ to the partition classes $A, B, C$.
The structure of Construction \ref{Bose} and Theorems \ref{CD} and \ref{EHcon} in turn implies the following sets of inequalities of Cauchy-Davenport and Erdős-Heilbronn type, respectively.

\begin{equation}\label{Cauchytype}
\begin{split}
    |A^*|&\geq \min\{p,\ |-A_0|+|B_0|-1\} \mbox{ \ if \ } |A_0|, |B_0| > 0,\\
    |B^*|&\geq \min\{p,\ |-B_0|+|C_0|-1\} \mbox{ \ if \ } |B_0|, |C_0| > 0,\\
    |C^*|&\geq \min\{p,\ |-C_0|+|A_0|-1\} \mbox{ \ if \ } |C_0|, |A_0| > 0.
\end{split}\end{equation}

\begin{equation}\label{EHtype}
\begin{split}
        |A^*|&\geq \min\{p,\ 2|C_0|-3\},\\
    |B^*|&\geq \min\{p,\ 2|A_0|-3\},\\
    |C^*|&\geq \min\{p,\ 2|B_0|-3\}.
\end{split}\end{equation}
Note that in the Erdős-Heilbronn type inequalities (\ref{EHtype}), the lower bound can be improved by one if the set consists of a single element. We distinguish several cases according to the sizes of the sets $A_0, B_0$, and $C_0$.

First suppose that two of these partition sets are empty. In this case, one Erdős-Heilbronn type inequality (\ref{EHtype}) provides the desired bound.

Next suppose that exactly one of these sets, say $C_0$, is empty. Thus we may apply two Erdős-Heilbronn type and one Cauchy-Davenport type inequality to obtain
\begin{multline}\notag
|(A^*\setminus A_0)\cup (B^*\setminus B_0)\cup (C^*\setminus C_0)|\geq \\  \min\{p,\ |A_0|+|B_0|-1\}-|A_0|+\min\{p,\ 2|A_0|-3\}-|B_0|+\min\{p,\ 2|B_0|-3\}.
\end{multline}

Hence it is enough to show  that
$$ \min\{p,\ |A_0|+|B_0|-1\}+\min\{p,\ 2|A_0|-3\}+\min\{p,\ 2|B_0|-3\}\geq 2(|A_0|+|B_0|)-3$$
holds when both sets consist of at least two elements, otherwise the proof is straightforward. Then, depending on the relation between  $p$, $|A_0|$ and $|B_0|$, we may apply either $3p\geq 2(|A_0|+|B_0|)$ which comes from $|V_0|\le \frac{n}{2}$ or $p\geq \{ |A_0|, |B_0|\} \geq 2$ to get the desired bound.

Finally, suppose that none of  $A_0, B_0$,  $C_0$ are empty, i.e., we can apply all the inequalities of (\ref{Cauchytype}) and (\ref{EHtype}). In order the finish the proof, consider the following proposition, the proof of which is straightforward.

\begin{prop}\label{seged}
Suppose that $z \geq \min\{p,\ q_1\}$ and $z \geq \min\{p,\  q_2\}$ holds for $z, q_1, q_2\in \mathbb{Z}$. Then $$z\geq \min\{p,\ \lceil\lambda q_1+(1-\lambda)q_2\rceil\}$$ also holds for $\lambda\in [0,1]$.
\end{prop}

We apply Proposition \ref{seged} where $|A^*|, |B^*|$ and $|C^*|$ takes the role of $z$ with the corresponding lower bounds of (\ref{Cauchytype}) and (\ref{EHtype}) and $\lambda=\frac{1}{3}$,  which provides

\begin{equation}\label{finally}
\begin{split}
 |A^*|&\geq \min\left \{p,\ \frac{1}{3}(2|C_0|-3)+\frac{2}{3}(|A_0|+|B_0|-1)\right \}  \\
 |B^*|&\geq \min\left \{p,\ \frac{1}{3}(2|A_0|-3)+\frac{2}{3}(|B_0|+|C_0|-1)\right \} \\
  |C^*|&\geq \min\left \{p,\ \frac{1}{3}(2|B_0|-3)+\frac{2}{3}(|C_0|+|A_0|-1)\right \}
\end{split}\end{equation}

By summing them up, this would imply a slightly weaker bound $$|A^*\cup B^*\cup C^*|\geq 2(|A_0|+|B_0|+|C_0|)-5.$$
However, it is impossible to have equality in all the inequalities of (\ref{finally}). Indeed, suppose that $C_0$ has the least size among the three sets $A_0, B_0$, $C_0$.  Then we could have use a better lower bound $(|A_0|+|B_0|-1)$ for $|A^*|$  instead of $\frac{1}{3}(2|C_0|-3)+\frac{2}{3}(|A_0|+|B_0|-1)$ in the first line of (\ref{finally}), which would yield an improvement of at least $\frac{4}{3}$   except when $|C_0|\geq |A_0|-1$ and $|C_0|\geq |B_0|-1$  moreover one of these inequalities is strict, say the one corresponding to $B_0$. But in the latter exceptional case, we still get an improvement of $\frac{2}{3}$ corresponding to $|A^*|\geq \min\{p,\ \frac{1}{3}(2|C_0|-3)+\frac{2}{3}(|A_0|+|B_0|-1)\} $, and we similarly get another improvement of $\frac{2}{3}$ corresponding to $|C^*|\geq \min\{p,\ \frac{1}{3}(2|B_0|-3)+\frac{2}{3}(|C_0|+|A_0|-1)\}$, as $B_0$ is a set of least size among the three sets $A_0, B_0$, $C_0$ as well. Thus by taking the ceiling, we get the desired bound.
\end{proof}

\section{Spreading linear triple system}

\subsection{Proofs -- lower bounds}

Doyen \cite{Doyen} proved the existence of spreading Steiner triple systems for every admissible order $n$, and applied the name \emph{non-degenerate plane} for such systems. In this section, we investigate how much sparser  a linear triple system can be to keep its spreading property. It follows immediately that such a system $\mathcal{F}$ should be dense enough compared to a $\STS(n)$. Indeed, the complement of the shadow $G(\mathcal{F})$ must be triangle-free, which in turn implies   $ \frac{1}{12}n^2< |\mathcal{F}|$ according to the theorem of Mantel and Turán.

\begin{proof}[Proof of Theorem \ref{spreading}, lower bound]

Our aim is to obtain an upper bound on $E(\overline{G})$, the number of edges not covered by the triples of a linear spreading system that is denoted by $\mathcal{F}$.
We start with three simple observations. Suppose that $|V(\mathcal{F})|>5$. Then
\begin{enumerate}[(1)]
\item $\overline{G}$ does not contain $K_3$.

\item For every subgraph $K_{1,3}$ in $\overline{G}$, the leaves cannot determine a triple of $\mathcal{F}$.

\item For every pair of triples of $\mathcal{F}$ which share a vertex, the corresponding $5$-vertex graph in $\overline{G}$ cannot contain more than $3$ edges.
\end{enumerate}

These statements follow from the definition of spreading, i.e. that apart from the triples of  the system, there are no subsets $V_0\subset~ V(\mathcal{F})$ coinciding with their closure, of size $|V_0|=3,4,5$, respectively.


Let $F$ denote a $4$-vertex subgraph of the shadow $G$ obtained from a triple $T$ of $\mathcal{F}$ and a vertex adjacent to exactly one vertex of the triple in $G$. Such a vertex is called the {\em private neighbour} of $T$.
Counting the pairs of edges of $\overline{G}$, we get that the number of $F$ subgraphs of $G$ is
$$\sum_v \binom{\overline{d}(v)}{2}$$
where $\overline{d}$ denotes the degree function on the vertices of $\overline{G}$.
Indeed, every such pair of adjacent non-edges $vu, vu'$ spans an edge hence determines the triple $\{u, u',u'' \}$ by observation (1), and $vu''$  must be an edge in $G$ in view of observation (2).

On the other hand, every subgraph $F$ can be determined by a triple  $T$ and one of its private neighbours. Let the {\em value of the triple $T$}, $\Val(T)$ denote the number of private neighbours of  the triple $T$, i.e., the number of $F$ subgraphs corresponding to the triple.
 We thus obtain

 \begin{equation}\label{eq:4}
\sum_v \binom{\overline{d}(v)}{2} = \sum_{T \in \mathcal{F}}\Val(T). \end{equation}

Observe that $|\mathcal{F}|= \frac{1}{6}(2\binom{n}{2}-\sum_v  {\overline{d}(v)})$, moreover $\Val(T)\leq n-3$  clearly holds for every triple $T$. By the application of the bound $\Val(T)\leq n-3$, one  would directly derive $E(\overline{G})\leq \frac{\sqrt{13}-1}{12}n^2+~O(n)\approx 0.21n^2$ from Equation \ref{eq:4}, using the AMQM inequality. However, this upper bound on $\Val(T)$ cannot be sharp for every triple: if the value of a triple is much larger than $\frac{n}{2}$, then many triples have value less than $\frac{n}{2}$. To better understand this situation, take a triple $T=\{v_1, v_2, v_3\}$, and denote by $N^*_i$ the vertices which are adjacent only to $v_i$ from the triple  $\{v_1, v_2, v_3\}$, for $i\in \{1,2,3\}$.

\begin{obs}\label{simple}   
$G[N^*_1 \cup N^*_2 \cup N^*_3]$ is a complete graph.
\end{obs}

\begin{proof}
Indeed, since every pair of vertices  from this class has a common non-neighbour, thus they must be joined in $G$ to avoid a $K_3$ in $\overline{G}$.
\end{proof}

Now we define a new graph $\mathcal{G}= \mathcal{G}(\mathcal{F})$ as follows: we assign a vertex to every triple $T \in  \mathcal{F}$, and we join $T$ and $T'$ if a pair from each span a $C_4$ in $\overline{G}$. Note that Observation (2) implies that such a pair is unique if $T\sim T'$. These vertex pairs are called the {\it supporting pairs} of the $C_4$, and denoted by $\Supp(T, T')$ and $\Supp(T', T)$ for the pair in $T$ and in $T'$, respectively.

\begin{prop}\label{vals}
Suppose that $T\sim T'$ in $\mathcal{G}$. Then $\Val(T)+\Val(T')\leq n$.
\end{prop}

\begin{proof}
Without loss of generality, we may suppose by observation (2) that $T=\{v_1, v_2, v_3\}$, $T'\supset\{u,w\}$, and $\{u,w\}\subset N^*_1$. Observe that $\Val(T)= |N^*_1 \cup N^*_2 \cup N^*_3|$. On the other hand, Observation \ref{simple} implies that each vertex of the private neighbourhood set $N^*_1 \cup N^*_2 \cup N^*_3$ is adjacent in $G$ to at least $2$ vertices of $T'$, hence  $\Val(T')\leq n-\Val(T)$.
\end{proof}

We partition the vertex set of $\mathcal{G}$ to vertices with  $\Val(T)\geq \frac{n}{2}$  (class $A$) and  with $\Val(T)< \frac{n}{2}$  (class $B$). Consider now the bipartite graph $\mathcal{G}[A,B]$. We obtain lower and upper bounds on the degrees $\deg(T)$ $(T\in A \cup B)$ in this auxiliary bipartite graph as follows.

\begin{prop}\label{degrees}
$$\deg(T)\geq 3\binom{\frac{1}{3}\Val(T)}{2} {\mbox \ if\ } T \in A,$$

$$\deg(T')\leq \binom{n-\Val(T)-1}{2} {\mbox \ if\ } T' \in B,  \ T\sim T'.$$
\end{prop}

\begin{proof} 
To prove the first bound, we apply Proposition \ref{vals}, observation (2) and Observation \ref{simple}.  Note that every neighbour of $T=\{v_1, v_2, v_3\}$  in $\mathcal{G}[A,B]$ corresponds to a  pair of vertices  in one of the sets $N^*_i$ $(i=1,2,3)$ that supports a $C_4$ in $\overline{G}$, so $$\deg(T)= \sum_{i\in \{1,2,3\}}\binom{|N^*_i|}{2}\geq 3\binom{\frac{1}{3}\Val(T)}{2}$$ by Jensen's inequality.

To prove the second bound, observe that if $T'$ and  an arbitrary triple $T''$ span a $C_4$ $v_1'v_1''v_2'v_2''$ in $\overline{G}$ while  $T'$ and $T$ also span a $C_4$ in $\overline{G}$ then the supporting pair $\Supp(T'', T')=v_1''v_2''$ of the first $C_4$ must be disjoint from $\bigcup_{i}N^*_i$.

Indeed,  the vertices of $\Supp(T', T)$  are private neighbours of $T\setminus \Supp(T, T')$ in view of Observation (2). On the other hand, $\Supp(T', T'')$ has a common vertex with $\Supp(T', T)$, thus the vertices of  $\Supp(T'', T')$ are not connected to a vertex of $\bigcup_{i}N^*_i$. However, 
 $\bigcup_{i}N^*_i$ induces a complete graph according to Observation \ref{simple},
thus every triple $T''$ which is joined to $T'$ in the auxiliary graph $\mathcal{G}$ has pair of points (i.e. a skeleton edge) outside $\bigcup_{i}N^*_i$. Hence the number of neighbour triples $T''$ is at most the number of possible supporting pairs $\Supp(T'', T')$, bounded above by $\binom{n-\Val(v)}{2}$. Moreover, if $\Supp(T, T')=v_1, v_2$, then  $T\setminus\Supp(T, T')=: v_3$ must be a private neighbour of two vertices from $T'$, namely $\Supp(T', T)$ due to Observation (2).
But this implies that supporting pairs $\Supp(T'', T')$ cannot contain  $v_3$ either,
which completes the proof.
\end{proof}

\bigskip
Proposition \ref{vals} and \ref{degrees} enable us to improve the upper bound on the average value of the triples beyond $\Val(T)\leq n-3$. This is carried out in the following lemma.
\smallskip

\begin{lemma}\label{average}
Suppose that a weighted bipartite graph $\mathcal{G}(A,B)$ is given under the set of conditions
\begin{itemize}
    \item $\Val: A \rightarrow [\frac{n}{2},n]$ and $\Val: B \rightarrow [0, \frac{n}{2})$ holds for  the weight function;
    \item  $ \Val({v})+\Val({v'})\leq n \ \forall vv' \in E(\mathcal{G})$ ;
    \item  $\deg(v)\geq 3\binom{\frac{1}{3}\Val(v)}{2}$ if $v \in A$;
\item
$\deg(v')\leq \binom{n-\Val(v)-1}{2} {\mbox \ if\ } v' \in B,  \ vv'\in E(\mathcal{G}).$
\end{itemize}
Then \begin{equation}\label{eq:5}
 \sum_{v\in V(\mathcal{G})} \Val({v})\leq \tau n\cdot|V(\mathcal{G})|,
\end{equation}
where $\tau\approx 0.51829$ is the unique local extremum of the rational function $\frac{z(1-z)(3-2z)}{4z^2-6z+3}$ in the interval $z\in[\frac{1}{2}, 1]$.
\end{lemma}
\smallskip

We finish the proof by applying Lemma \ref{average}, and then return to the proof of Lemma \ref{average}.  Equality (\ref{eq:4}) and the bound (\ref{eq:5}) together gives

\begin{equation}
\sum_{v\in V(G)} \binom{\overline{d}(v)}{2} = \sum_{T \in \mathcal{F}}\Val(T)< 0.5183\cdot n|\mathcal{F}|
\end{equation}

On the other hand, since $|\mathcal{F}|= \frac{1}{6}(2\binom{n}{2}-\sum_{v\in V(G)} {\overline{d}(v)})$, this provides

$$n\left(\frac{\sum_{v\in V(G)} \overline{d}(v)}{n} \right)^2+\left(\frac{0.5183n}{3}-1\right)\sum_{v\in V(G)} \overline{d}(v) \leq \frac{0.5183}{3}(n^3-n^2) $$ by the AMQM inequality. Introducing $E(\overline{G})=\frac{1}{2} \sum_{v\in V(G)} \overline{d}(v)$, we get a quadratic inequality for $E(\overline{G})$ in terms of $n$, which gives the desired bound $E(\overline{G})<0.169n^2+O(n)$.
\end{proof}

\begin{proof}[Proof of Lemma \ref{average}]
Instead of considering it as an involved convex optimisation problem, the general idea is to obtain a biregular bipartite graph in which the vertices have larger average value and optimise the average in the class of biregular bipartite graphs. The proof is carried out in three main steps.

First take a vertex $v_0$ of maximal value. We claim that for all of its neighbours $v'\in B$, the inequalities corresponding to them in Lemma \ref{average} would hold with  equalities:
\begin{enumerate}[(i)]
   \item  $ \Val({v'})= n-\Val({v_0})$,
\item $\deg(v')= \binom{n-\Val(v_0)-1}{2},$
\end{enumerate}
 or else the average value could be increased. The claim for (i) is straightforward, while for (ii) suppose that $v'\in N(v_0)$ has smaller degree. Then one could take $\lceil3\binom{\frac{1}{3}\Val(v_0)}{2}\rceil$ disjoint copies of $\mathcal{G}$, add a new vertex $v_0^*$ (of value $\Val(v_0)$) and join to every copy of $v'$. Hence the conditions were fulfilled, while the average value would be increased.

Similar argument shows that for each $u\in A$ for which $N(v_0)\cap N(u)\neq \emptyset$, $\deg(u)=\lceil 3\binom{\frac{1}{3}\Val(u)}{2}\rceil$. Suppose it is not the case. Then for any $v'\in N(v_0)\cap N(u)$ one could delete the edge $uv'$ in $\mathcal{G}$, then take $\lceil 3\binom{\frac{1}{3}\Val(v_0)}{2}\rceil$ disjoint copies of the derived graph and finally add a new vertex $v_0^*$ (of value $\Val(v_0)$) and join to every copy of $v'$.

Without loss of generality we can assume that for each $u\in A$ for which $|N(v_0)\cap N(u)|=\lambda_u > 0$ with a maximum value vertex $v_0$,  every neighbour $v'$ of $u$ is adjacent to a vertex of maximum value. Consider the following construction. We take $m\cdot \deg(u)$ disjoint copies of $\mathcal{G}$ for an arbitrarily chosen $m\in \mathbb{Z}^+$ and redistribute the neighbours of the copies of $u$ in such a way that $m\cdot \lambda_u$ copies of $u$ are each joined to  $\deg(u)$ distinct vertices from the copies of $N(v_0)\cap N(u)$, and the rest of the copies of $u$ are each joined to  $\deg(u)$ distinct vertices from the copies of $N(u)\setminus N(v_0)$.  Since $m$ can be chosen arbitrarily, this step can be performed at the same time for each such vertex $u$ (as $m$ can be chosen as the least common multiple of all of the corresponding degrees). 

In order to maximize the average value of the vertices, we can clearly delete all but one connected components of the graph, and hence by the argument in the last paragraph we assume that every vertex $v'\in B$ is connected to a vertex of maximum value. Now let us rewrite the average value  as

$$\frac{1}{|V(\mathcal{G})|} \sum_{v\in V(\mathcal{G})} \Val({v})=\frac{1}{|V(\mathcal{G})|}\sum_{v\in A}\left(\Val(v)+\sum_{v'\in N(v)}\frac{\Val(v')}{\deg(v')}\right) .$$
Observe that the contribution of each vertex $v\in A$ to the average in the weighted sum can be measured by the proportion of the  weighted values corresponding to $v$ and the sum of the weights corresponding to $v$, that is,  $$\frac{\Val(v)+\sum_{v'\in N(v)}\frac{\Val(v')}{\deg(v')}}{1+\sum_{v'\in N(v)}\frac{1}{\deg(v')}}.$$

We show that these values must be maximal in the maximum value of the average, or else one could make an improvement.
According to our previous considerations, we may assume that  for all $v'\in B$, we have $\Val(v')=n-\Val(v_0)$ and moreover $\deg(v')= \binom{n-\Val(v_0)-1}{2}$.
In order to show that all the vertices of $A$ have the same degree
 we may compare the corresponding contributions of a vertex $v_0$ of maximum value and some other vertex $u\in A$ which has the second largest value.

 Clearly either
 $$\frac{\Val(v_0)+ (n-\Val(v_0))\frac{\deg(v_0)}{\deg(v')}}{1+\frac{\deg(v_0)}{\deg(v')}}\geq \frac{\Val(u)+ (n-\Val(v_0))\frac{\deg(u)}{\deg(v')}}{1+\frac{\deg(u)}{\deg(v')}},$$
or
 $$\frac{\Val(v_0)+ (n-\Val(v_0))\frac{\deg(v_0)}{\deg(v')}}{1+\frac{\deg(v_0)}{\deg(v')}}< \frac{\Val(u)+ (n-\Val(v_0))\frac{\deg(u)}{\deg(v')}}{1+\frac{\deg(u)}{\deg(v')}}.$$

 In both cases once again we can apply the above argument of copying the graph, eliminating a vertices of a certain degree (namely $\deg(v_0)$ or  $\deg(u)$) and redistributing its neighbourhood between {\it new vertices} of another fixed degree (namely $\deg(u)$ or  $\deg(v_0)$, respectively). The number of copies of the graph is chosen in terms of $\deg(v_0)$   $\deg(u)$, such that each vertex in partite class $B$ gets the same number of vertices, which is $\deg(v')= \binom{n-\Val(v_0)-1}{2}$, by joining them to new vertices instead of the deleted ones.

 This way we eliminate  either the vertices of maximum degree or of second maximum degree, while the average value is monotonically increasing. Doing so repeatedly,  after a  suitable number of steps we end up with a  bipartite graph where all vertices $v\in A$ have the same degree.

The argument implies that in order to determine the maximum of the average value under the constraints of Lemma \ref{average},
it is enough to determine the maximum average value in the class of biregular subgraphs as the third step to finish the proof.

To this end, consider the maximum of the function

$$w\rightarrow \frac{w\binom{n-w-1}{2}+(n-w)3\binom{\frac{1}{3}w}{2}}{\binom{n-w-1}{2}+3\binom{\frac{1}{3}w}{2}}$$  in the interval $w\in [\frac{n}{2},n]$, which is an equivalent reformulation of the problem.
Introducing $z=\frac{w}{n}$, we obtain the function

$$z \rightarrow n\frac{z(1-z)(3-2z-\frac{12}{n})+\frac{6z}{n^2}}{4z^2-6z+3-\frac{6}{n}(1.5-z-\frac{1}{n})}$$ on the domain $z\in[\frac{1}{2}, 1]$. One can verify that $$
 n\frac{z(1-z)(3-2z-\frac{12}{n})+\frac{6z}{n^2}}{4z^2-6z+3-\frac{6}{n}(1.5-z-\frac{1}{n})}\leq  n\frac{z(1-z)(3-2z)}{4z^2-6z+3}$$ holds   for $z\in[\frac{1}{2}, 1]$ which in turn implies the statement of the lemma.
\end{proof}

\begin{proof}[Proof of Theorem \ref{weaklysp}, lower bound]
Take an arbitrary triple $T_1$ of the weakly spreading system $\mathcal{F}$. Observe that there must exist a triple $T_2$ sharing a common vertex with $T_1$, otherwise the union of $T_1$ and any other triple would violate the weakly spreading property. From now on, the weakly spreading condition guarantees the existence of an ordering of the triples $T_1, T_2, \ldots T_m$ of $\mathcal{F}$, such that $$|T_k\cap \bigcup_{i=1}^{k-1} T_i|\geq 2\ \ (\forall k: 3\leq k\leq m).$$ This in turn implies the lower bound. By case analysis one can prove that it is sharp for $5\le n\le 10$.
\end{proof}

\subsection{Upper bounds -- construction for sparse spreading systems}

We will construct a spreading triple system $\mathcal{F}$ on $n=6p+3$ vertices for every  $p$ such that $p$ is an odd prime number, with $|E(G(\mathcal{F}))|\approx \frac{5}{12}n^2$. 

\begin{construction}\label{csodas}
The vertex set of $\mathcal{F}$ is the disjoint union of $6$ smaller subsets (we refer to them as classes), namely $V= A\cup B\cup C \cup A' \cup B' \cup C'$, where $|A|= |B| = |C| =p+1$ and $|A'| = |B'| = |C'| = p$. Denote the elements of $A$ with $a_0,a_1,\dots,a_{p-1}$ and a special vertex $a$. Similarly $B=\{b_0,b_1,\dots,b_{p-1},b\}$ and $C=\{c_0,c_1,\dots,c_{p-1},c\}$.
For $A',B',C'$  we note the corresponding vertices by $\alpha,\beta,\gamma$ respectively, and index their elements again from $0$ up to $p-1$. The set of triples in $\mathcal{F}$ are defined as follows:
\begin{itemize}
\item black triples:
\begin{itemize}
\item between $A$ and $B'$:~$\{a,a_j,\beta_j \}$ (for $0 \le j \le p-1$); and $\{a_i,a_{2j-i \pmod{p}}, \beta_{j}\}$ (for $0 \le i\ne j \le p-1$)
\item between $B$ and $C'$:~$\{b,b_j,\gamma_j \}$ (for $0 \le j \le p-1$); and $\{b_i,b_{2j-i \pmod{p}}, \gamma_{j}\}$ (for $0 \le i\ne j \le p-1$)
\item between $C$ and $A'$:~$\{c,c_j,\alpha_j \}$ (for $0 \le j \le p-1$); and $\{c_i,c_{2j-i \pmod{p}}, \alpha_{j}\}$ (for $0 \le i\ne j \le p-1$)
\end{itemize}
\item brown triples:
\begin{itemize}
\item between $A'$ and $B$:~$\{\alpha_i,\alpha_{2j-i \pmod{p}},b_j \}$ (for $0 \le i\ne j \le p-1$)
\item between $B'$ and $C$:~$\{\beta_i,\beta_{2j-i \pmod{p}},c_j \}$ (for $0 \le i\ne j \le p-1$)
\item between $C'$ and $A$:~$\{\gamma_i,\gamma_{2j-i \pmod{p}},a_j \}$ (for $0 \le i\ne j \le p-1$)
\end{itemize}
\item orange triples:
\begin{itemize}
\item between $A\setminus \{a\}$, $B\setminus \{b\}$ and $C\setminus \{c\}$:~ $\{a_i,b_j,c_{i+j \pmod{p}} \}$ (for $0 \le i,j \le p-1$)
\item between $A'$, $B'$ and $C'$:~$\{\alpha_i,\beta_j,\gamma_{i+j+1 \pmod{p}} \}$ (for $0 \le i,j \le p-1$)
\item $\{a,b,c\}$
\end{itemize}
\item red triples: \newline
$\{ a,\alpha_j,b_j \}$, $\{ b,\beta_j,c_j \}$ and $\{ c,\gamma_j,a_j \}$ (for $0\le j \le p-1$)
\item blue triples: \newline
$\{ a,\gamma_j,c_j \}$, $\{ b,\alpha_j,a_j \}$ and $\{ c,\beta_j,b_j \}$ (for $0\le j \le p-1$)
\end{itemize}
\end{construction}


\begin{figure}[t!]
    \centering
    \begin{subfigure}[b]{0.495\textwidth}
       \includegraphics[width=\linewidth]{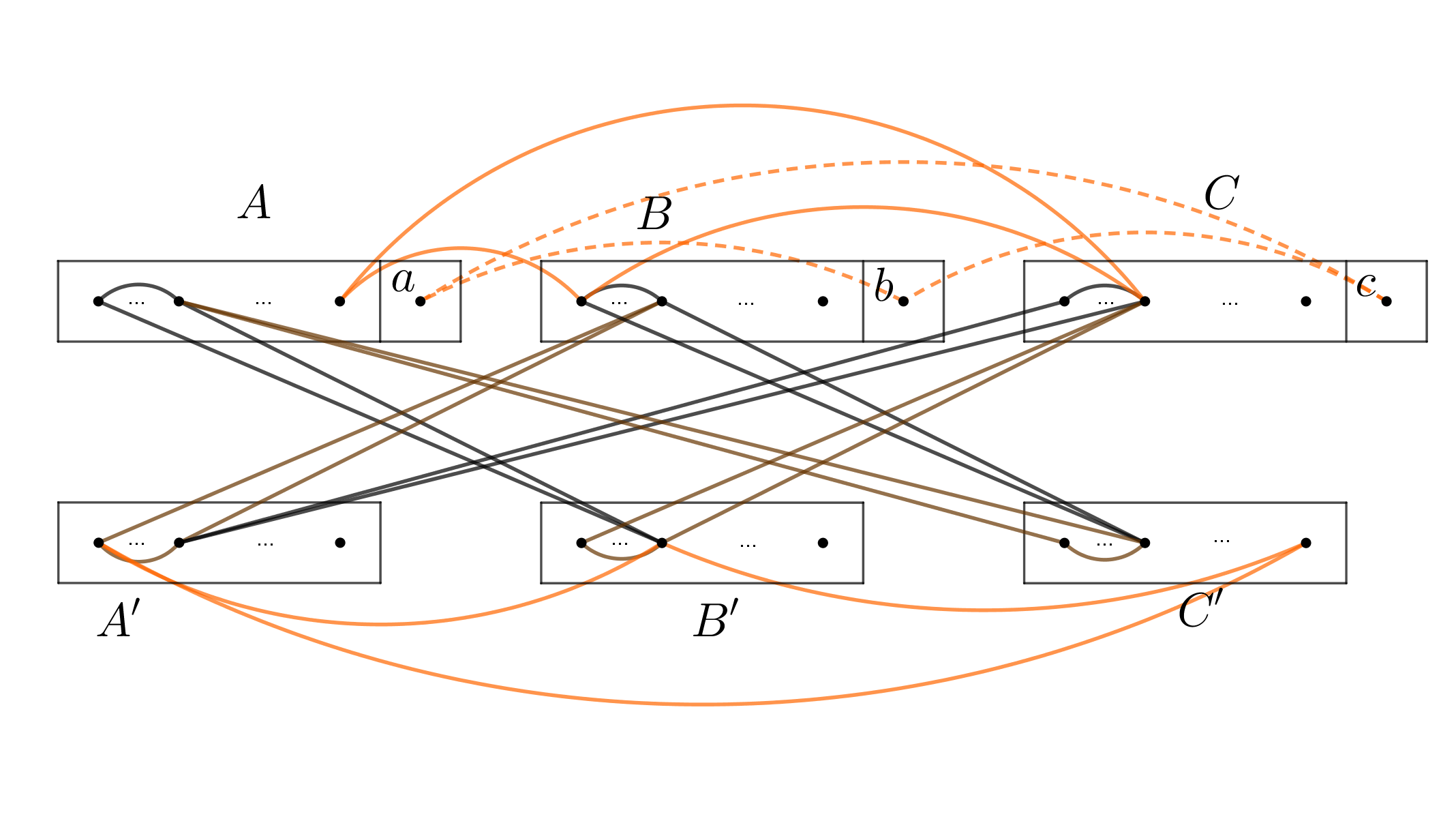}
    \caption{Black, brown and orange triples and $\{a,b,c\}$}
    \label{fig:bbo}
    \end{subfigure}
    \begin{subfigure}[b]{0.495\textwidth}
        \includegraphics[width=\linewidth]{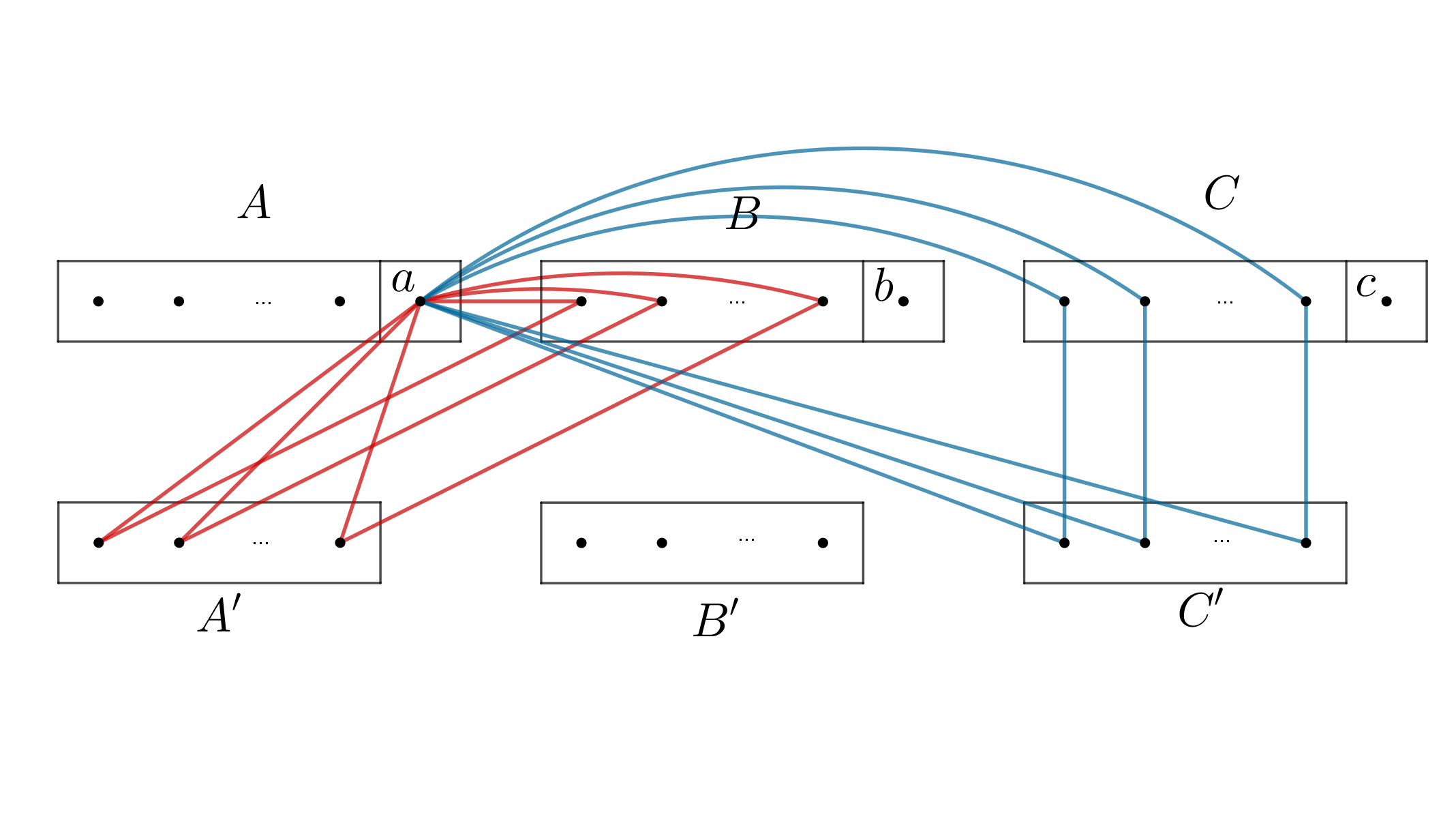}
    \caption{Red and blue triples through $a$}
    \label{fig:rb}
    \end{subfigure}
 \caption{Overview of the triple types in Construction \ref{csodas}}\label{fig}
\end{figure}

\begin{prop}\label{spreads}
The triple system $\mathcal{F}$ defined above has the spreading property.
\end{prop}

The proof requires some case analysis. The main goal is to show that $\cl(V')~\neq~ V'~\subset~V(\mathcal{F})$ for nontrivial subsets. Our tool is the Cauchy--Davenport theorem (see Theorem \ref{CD}), which simplifies the analysis to those cases where apart from the special vertices $a,b,c$, we have $|V'\cap (A\cup B\cup C)|\leq 3$ and  $|V'\cap (A'\cup B'\cup C')|\leq 3$. For the sake of completeness the proof is carried out in the Appendix.

We continue with a construction which shows a linear upper bound on the minimum size of weakly spreading systems. This will be derived from the upper bound of  Theorem \ref{spreading} and completes the proof of
the upper bound of Theorem \ref{weaklysp}.

\begin{construction}[Crowning construction]\label{small_rec}
Consider a linear spreading system $\mathcal{F}$ on $n$ vertices and $\xi_{sp}(n)=\frac{1}{3}\left( \binom{n}{2}-Cn^2 \right)$ triples, with the appropriate constant $C$. Assign a new vertex $v(xy)$ to every not-covered edge $xy$ of the underlying graph $G=G(\mathcal{F})$, and add newly formed triples by taking  $\{ \{x,y,v(xy)\}: xy \not \in G\}$.
\end{construction}

\begin{prop}\label{wsp_bound}
Construction \ref{small_rec} provides a weakly spreading system on $n+Cn^2$ vertices with $\frac{1}{3}\left( \binom{n}{2}+2Cn^2 \right)$ triples, hence we obtain
$$\xi_{wsp}(N)\leq \frac{2}{3}N+\frac{1}{6C}N.$$ 
The deletion of an arbitrary subset of the new vertices of form $v(xy)$ provides also a weakly spreading system.
\end{prop}
\begin{proof}
It is easy to verify that any two triples, whose underlying set is denoted by $V'$, determine at least three vertices which are not newly added such that they do not form a triple in $\mathcal{F}$. By the spreading property of $\mathcal{F}$, we get that $\cl(V')$ contains all points besides the new ones. Through the newly formed triples we get that actually every vertex is contained in $\cl(V')$.
\end{proof}

\begin{proof}[Proof of Theorem \ref{weaklysp}, upper bound] 
Let us choose the smallest odd prime $p$ such that the sum of the number of vertices and uncovered edges in Construction \ref{csodas} is at least $N$. Consider Construction \ref{small_rec} built on this spreading system and delete a subset of new vertices to obtain a triple system on $n$ vertices. The bound thus follows from the application of  Proposition \ref{wsp_bound} with $C=\frac{1}{12}$, if we take into account 
the result on the gaps between consecutive primes \cite{Pintz}, similarly to Corollary \ref{corall}. 
\end{proof}

\section{Related results and open problems}

In this section we point out several related areas. First we discuss the connection to the topic of Latin squares, a message of which is that  similar structures  often provide constructions for the problem in view.
\subsection{Latin squares}

A Latin square of order $n$ is an $n\times n$ matrix in which each one of $n$ symbols appears exactly once in every row and in every column.  A subsquare of a Latin square is a submatrix of the Latin square which is itself a Latin square.
Note that Latin squares of order $n$ and $1$-factorizations of  complete bipartite graphs $K_{n,n}$ are corresponding objects.
We will apply the following theorem due to Maenhaut, Wanless and Webb \cite{Wanless}, who were building on the work of Andersen and Mendelsohn \cite{AM}.

\begin{theorem}[Maenhaut, Wanless and Webb, \cite{Wanless}] Subsquare-free Latin squares exists for every odd order.
\end{theorem}

Note that for prime order the statement follows from the Cauchy--Davenport theorem, since the Cayley-table of group of prime order gives an instance. The construction presented below not only gives a simple weakly spreading construction, but it
may provide an ingredient to a possible extension of Construction \ref{csodas}, where the triangle decomposition of  the balanced complete tripartite graph, denoted by the set of orange triples, were obtained by a Cauchy--Davenport argument in the prime order case.

\begin{const} \label{latinos}
Take a subsquare-free Latin square of odd order $n$ with row set $U$, column set $V$ and symbol set $W$. We assign a triple system $\mathcal{T}$ on $U\cup V\cup W$ to the Latin square  as follows.   Let $T=\{u_i, v_j, w_k\}\in \mathcal{T}$ if and only if $w_k$ is the symbol in position $(i,j)$ in the Latin square.
\end{const}

\begin{prop}
 Construction \ref{latinos} yields that the minimum size of a weakly spreading triple system is at most $\xi_{wsp}(n)\leq \frac{n^2}{9}$ for $n\equiv 3 \pmod 6$.
\end{prop}

\begin{proof} Observe first that every pair of elements from different classes is contained exactly once in the system $\mathcal{T}$. Thus we have to show that there does not exist a subsystem of $\mathcal{T}$ spanned by $U'\subseteq U$, $V'\subseteq V$ and $W'\subseteq W$ for which every pair or elements from different classes is contained exactly once in triple of the subsystem. Clearly the existence would only be possible if  $1<|U'|=|V'|=|W'|<n$ but such a system  would be equivalent to a Latin subsquare, a contradiction.
\end{proof}

\subsection{Influence maximization}
A social network represented  by the graph of relationships and interactions in a group of individuals plays a fundamental role as a medium for the spread of information, ideas, and influence among its members.
Models for the processes where some sort of influence or information propagate through a social network have been studied in a number of domains, including sociology, psychology, economics and
computer science.  The influence of a set  of nodes
is the (expected) number of active nodes at the end of the propagation process in the model and the influence maximization question asks for the set of given size which has the largest influence.
In one of the models, called  the threshold model (cf. \cite{Chen}) there exists a threshold value $t(v)$ for every vertex $v\in V$ and in each discrete step a vertex is activated if it has at least $t(v)$ active neighbours. For more details we refer to the recent surveys \cite{influ-survey, survey2} and to the pioneer papers of Domingos and Richardson \cite{Domi} and Kempe et. al. \cite{Kempe}.

Mostly in models of social networks one only considers the graph of relationships, however in applications the propagation may depend more on whether an individual is influenced by the majority of the group members of social groups he or she belongs to. In that context, one has to describe the groups as hyperedges of a hypergraph, and in case of linear $3$-graphs, the propagation of a vertex set $V$ would clearly influence its closure $\cl(V)$. Hence our results determine bounds on the number of $3$-sets needed so that every set of $3$ vertices besides the triples themselves, or every pair of triples has maximum influence.

\subsection{Connectivity, backward and forward $3$-graphs}

First we recall the concept of $k$-vertex-connectivity of hypergraphs (cf. \cite{connect}), which is strongly related to the properties in view, and introduce a new edge-connectivity concept for triple systems.

\begin{defn}
A hypergraph $\mathcal{F}$ is {\em $k$-vertex connected} if  $|V(\mathcal{F})|\geq k$ and there is
 no set $W\in V(\mathcal{F})$ of size at most $k-1$ such that its removal disconnects $\mathcal{F}$. $\mathcal{F}$ is disconnected if there exist a vertex partition $\{U,V\setminus U\}$  such that the triples of  $\mathcal{F}$ are spanned by either $ U$ or $V\setminus U$.

A $3$-uniform hypergraph $\mathcal{F}$ is {\em strongly connected} if  every vertex partition $\{U,V\setminus U\}$ induces a triple $T$ with $|T\cap U|=2$, provided $|U|\geq 4$.
\end{defn}

The latter definition implies that if the partition classes $U$ and $V\setminus U$  are large enough, then triples of type  $|T\cap U|=2$ and  $|T\cap U|=1$ both should appear. The condition $|U|\geq 4$ enables us to apply this concept for linear $3$-graphs. We note that the spreading property is stronger than the strong connectivity, while the weakly spreading property is weaker.

\begin{obs} A Steiner triple system is subsystem-free, that is, spreading
if and only if it is strongly connected. Every spreading linear triple system is strongly connected. Every strongly connected $3$-graph is weakly spreading.
\end{obs}

Finally, we underline that the weakly spreading property is not a local one, as   the condition $\cl(V')\supset V'$  restricted to every pair of triples,  $V'=V(\mathcal{F}')$ with $|\mathcal{F}'|=2$ by no means implies weakly spreading.
This follows from the construction below.

\begin{const}\label{teljes}
 Consider the complete graph  $K_n$ on $n$ vertices $n>3$, and add a vertex $v_{ij}$ to every  graph edge $v_iv_j$. The obtained triple system $\mathcal{F}_{(n)}=\{ \{v_i, v_j, v_{ij}\} | i\neq j\leq n\}$  on $\binom{n}{2}+n$ vertices with $\binom{n}{2}$ hyperedges  has the property that every pair of triples generate at least one further triple, but their closure will correspond to either $\mathcal{F}_{(3)}$ or $\mathcal{F}_{(4)}$.
\end{const}

We finish this subsection by mentioning  a connection to directed hypergraphs.
A directed hyperedge  is an ordered pair, $E=(X,Y)$, of disjoint subsets of vertices where $X$ is the tail while $Y$ is the head of the hyperedge. \emph{Backward, resp. forward $3$-graphs} are defined as directed $3$-uniform hypergraphs with hyperedges having a singleton head, resp. tail, see e.g. \cite{Gallo}. These objects have many applications in computer science, operations research, bioinformatics and transport networks.\\
It is easy to see that if one directs each triple of a linear $3$-graph in all possible three ways to obtain a backward edge, then the strong connectivity, described above,  of the triple system and the connectivity of the resulting  directed hypergraph are equivalent (cf. \cite{Gallo}).

\subsection{Further results and open problems}

We also mention  the recent  related work  of Nenadov, Sudakov and Wagner \cite{SudWag}
on embedding partial Steiner triple system to a small
complete $\STS$, and in general, embedding certain partial substructures to complete structures. In the spreading problem of linear $3$-graphs, one may consider the triples of the hypergraph as a collinearity prescription for triples of points. Under this condition the aim would be to describe those  affine or projective planes of given order $q$ in which the prescription can be fulfilled, i.e. the corresponding partial linear space can be embedded into the geometry in view. This is closely related to the problem of embedding partial geometries to a given geometry but we allow collinearity prescriptions for triples concerning the same line.


While our Theorem \ref{expander} on the expander property was sharp, our results Theorem \ref{spreading} and \ref{weaklysp} concerning spreading and weakly spreading  determined the corresponding parameter only up to a small constant factor. The authors believe that if $n$ is large enough, then neither of the bounds are sharp; however it seems a hard problem to asymptotically determine the exact values, similarly to many other extremal problems in hypergraph theory. We finish our paper with several open problems.

\begin{problem} Is the asymptotically best upper bound on the minimum size $\xi_{wsp}(n)$ of a linear weakly spreading triple system  obtained by the Crowning Construction \ref{small_rec} from an optimal construction for $\xi_{sp}(n)$?
\end{problem}

Although the lower bound  $\xi_{wsp}(n)$ is tight for $n\leq 10$, we conjecture that this might be the case, meaning that  $(C+o(1))n\leq \xi_{wsp}(n)$ for some $C>1$.

\begin{problem} Generalize the results to $r$-uniform (linear) hypergraphs $\mathcal{F}$.
\end{problem}
In order to do this, one should define the neighbourhood and closure accordingly: a vertex $z$ in the neighbourhood of $V'$, if and only if there exist a hyperedge $F\in \mathcal{F}$ containing $z$ such that
\begin{itemize}
    \item either $|F\cap V'|\geq \frac{r}{2}$ (majority rule)
    \item or $|F\cap V'|\geq t$, $t<r$ fixed  (large intersection).
\end{itemize}

\begin{problem} Prove the existence of a Steiner triple system $\STS(n)$ of arbitrary admissible order $n$, for which $${|N(V')|}\geq {|V'|}-3$$ for every  $V'\subset V(G)$ of size  $|V'|\leq \frac{|V|}{2}$.
\end{problem}

{\bf Acknowledgement} The authors are grateful to Kristóf Bérczi for his comments and for pointing out several related areas. We also gratefully thank the anonymous referees for their useful suggestions.

\newpage

\section{Appendix -- Proof of Theorem \ref{spreads}}

Here we prove that Construction \ref{csodas} has the spreading property, i.e. $\cl(V')=V$ for all nontrivial sets $V'$.
Note  that the construction corresponds to a triangle decomposition of a graph $G$ which  is a perturbation of $K_{6p}\setminus\left( K_{p,p}\dot\cup K_{p,p}\dot\cup K_{p,p}\right).$

The first step is to verify the statement for those sets $V'$ which have a large enough intersection with either $A\cup B\cup C$ or $ A'\cup B'\cup C'$, by the application of the Cauchy--Davenport theorem.

\begin{figure}[h!]
    \centering
    \begin{subfigure}[b]{0.495\textwidth}
       \includegraphics[width=\linewidth]{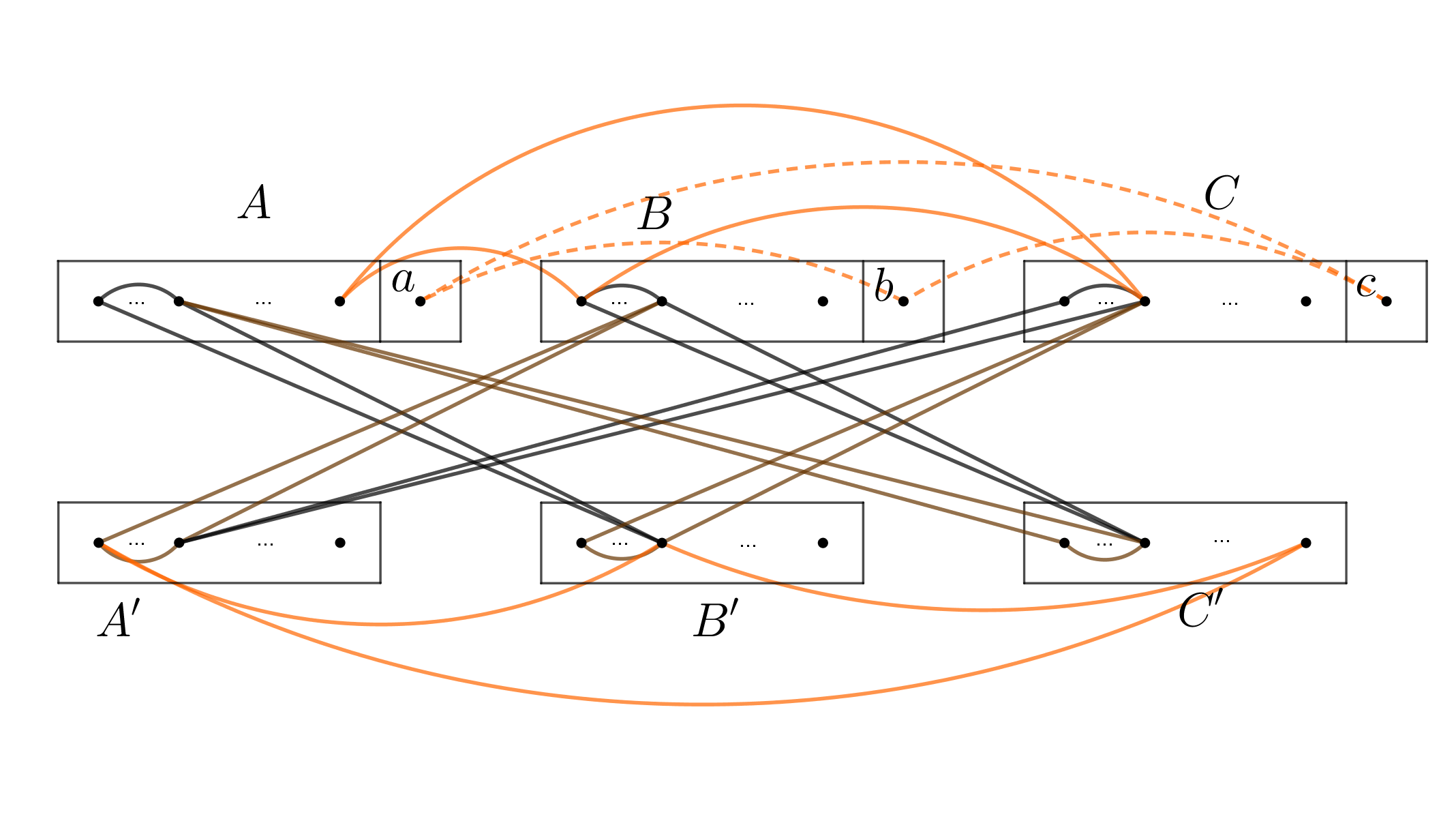}
    \caption{Black, brown and orange triples and $\{a,b,c\}$}
    \end{subfigure}
    \begin{subfigure}[b]{0.495\textwidth}
        \includegraphics[width=\linewidth]{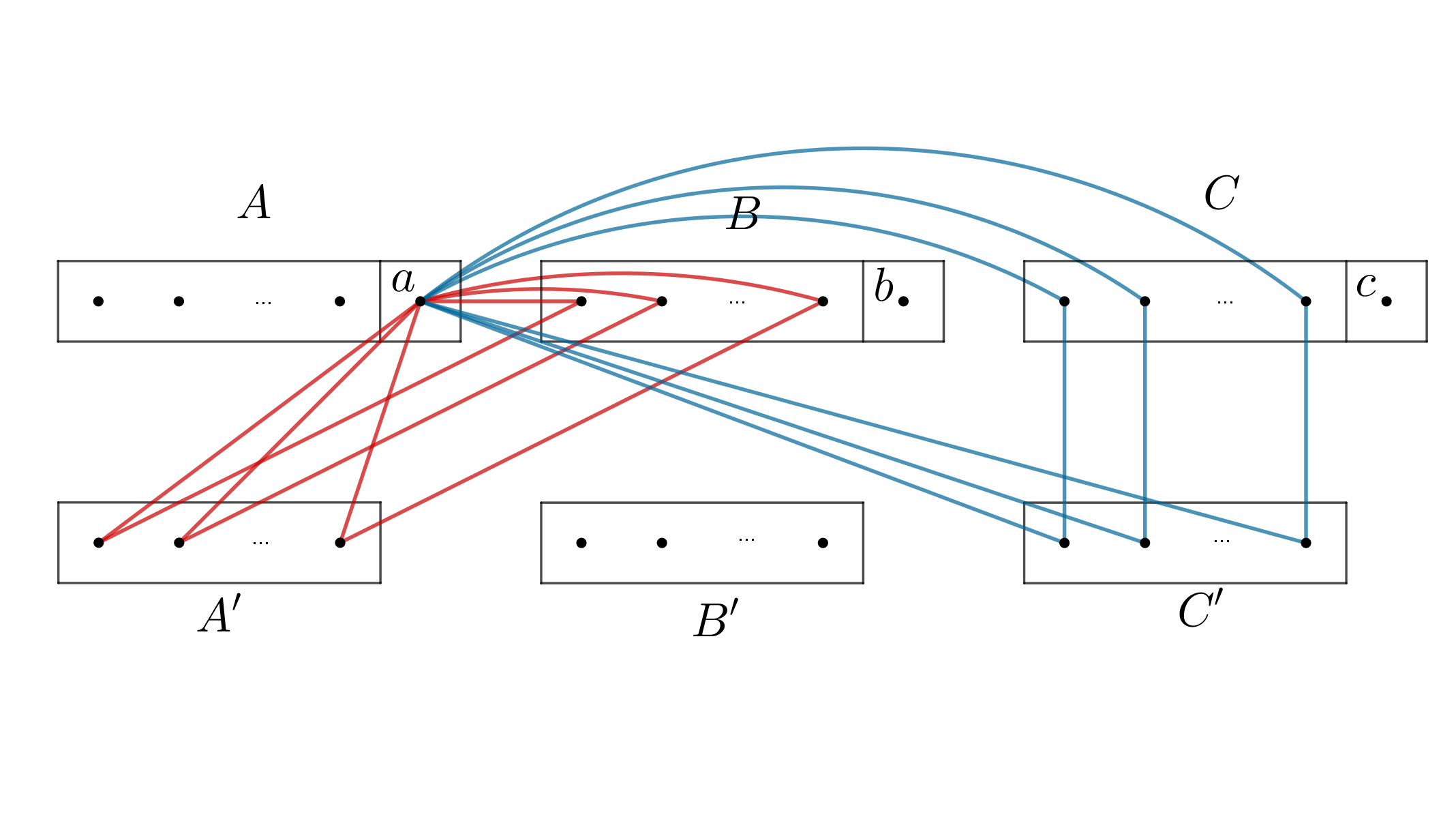}
    \caption{Red and blue triples through $a$}
    \end{subfigure}
 \caption{Overview of the triple types in Construction \ref{csodas}}
\end{figure}

\begin{obs}\label{4spread} Let us denote the set $(A\setminus\{a\}) \cup (B\setminus\{b\}) \cup (C\setminus\{c\})$ by $U$. If $|V'\cap U|> 3$ or  $|V'\cap (A'\cup B'\cup C')|> 3$, and $V'$ intersects with at least two different classes then $\cl(V')=V$.
\end{obs}

Indeed, without loss of generality, suppose  to the contrary that there exists a set $|V'\cap (A'\cup B'\cup C')|> 3$ with  $A_0=V'\cap A'$, $B_0=V'\cap B'$, $C_0=V'\cap C'$ from which $C_0$ has the least size (smaller than $p$), such that $\cl(V')=V'$.  Apply the Cauchy--Davenport theorem (Theorem \ref{CD}) for $A_0$ and $B_0$ to obtain
$|A_0+B_0|\geq \min\{|A_0|+|B_0|-1, p\}>|C_0|$. Thus the orange triples with their additive structure ensure that $|\cl(V')\cap C'|>|C_0|$, a contradiction.

In the rest of the proof we point out that no matter how we choose a nontrivial set $V'$ of $3$ elements $\{x,y,z\}$, its closure contains at least $4$ elements from either $U$ or $A'\cup B'\cup C'$, coming from more than one class, thus the application of Observation \ref{4spread} in turn completes the proof.

\begin{enumerate}
    \item $\{x,y,z\} \subset A \cup B \cup C$:
    \begin{enumerate}
        \item[a)] $|\{x,y,z\} \cap \{a,b,c\}|=0$:
        \begin{enumerate}
            \item[a1)] If the starting elements are not in the same class ($A$, $B$ or $C$) then two of them from different classes determine a new element (moreover it cannot be the special vertex) from the third class via an orange triple and now we have 4 elements of the closure in $U$ not from the same class.
            \item[a2)] WLOG we can assume that $x=a_i,y=a_j,z=a_k$ from $A\setminus\{a\}$. By using black and brown triples $a_i,a_j$ determine some $\beta_l$; $a_i,a_k$ determine some $\beta_m$ ($m\ne l$) and $\beta_l,\beta_m$ determine some $c_s \in C \setminus\{c\}$.
        \end{enumerate}
        \item[b)] $|\{x,y,z\} \cap \{a,b,c\}|=1$, WLOG let's assume that $z=a$:
        \begin{enumerate}
            \item[b1)] If $x,y$ are in different classes then they determine a new element from the third class via an orange triple thus we have got now 4 elements: $a_i,a,b_j,c_k$. From $a$ and $c_k$ we get $\gamma_k$ due to a blue triple. If $j\ne k$ then $b_j$ and $\gamma_k$ determine some $b_m$ through a black triple, and the closure meets $U$ in more than 3 elements. If $j=k$ then $i=0$ must hold, therefore $b$ and $\beta_0$ are in the closure from black triples formed by $\{b,b_k,\gamma_k\}$ and $\{a,a_0,\beta_0\}$. Now $b$ and $\beta_0$ determine $c_0$ via a red triple, and we are done unless $i=j=k=0$. In that case, one can verify that the closure contains $\{a_0,b_0,c_0,\alpha_0, \beta_0,\gamma_0, a, b, c\}$ and by $\alpha_0$ and $\beta_0$ we get that $\gamma_1$ is in the closure via an orange triple hence $\gamma_0$ and $\gamma_1$ determine  $a_{\frac{p+1}{2}}$ via a brown triple that is the fourth element from $U$.
            \item[b2)] If $x,y$ are in the same class then $a,x$ and $a,y$ determine different elements of the same class from $A'\cup B' \cup C'$ therefore these two elements determine a new element of $U$ hence we trace back to case a).
        \end{enumerate}
        \item[c)] $|\{x,y,z\} \cap \{a,b,c\}|=2$, WLOG let's assume that $z=a$ and $y=b$:
        \begin{enumerate}
            \item[c1)] if $x=c_i \in C \setminus\{c\}$ then $a$ and $c_i$ determine $\gamma_i$ via a blue triple, then $b$ and $\gamma_i$ determine $b_i \in B \setminus \{b\}$ due to a black triple therefore we trace back to case b1).
            \item[c2)] if $x=b_i \in B \setminus\{b\}$ then $b$ and $b_i$ determine $\gamma_i$ via a black triple, then $a$ and $\gamma_i$ determine $c_i \in C \setminus \{c\}$ therefore we trace back to case b1).
            \item[c3)] if $x=a_i \in A \setminus\{a\}$ then $b$ and $a_i$ determine $\alpha_i$ via a blue triple, then $a$ and $\alpha_i$ determine $b_i \in B \setminus \{b\}$ due to a red triple therefore we trace back to case b1).
        \end{enumerate}
    \end{enumerate}
    \item $\{x,y,z\} \subset A' \cup B' \cup C'$:

    One can deduce that this case can be discussed precisely the same way as case 1.a).

    \item $|\{x,y,z\} \cap (A \cup B \cup C)|=2$:

    Assume that $\{y,z\}\subset A \cup B \cup C$ and $x \in A' \cup B' \cup C'$.
    \begin{enumerate}
        \item[a)] $|\{y,z\} \cap \{a,b,c\}|=0$:
        \begin{enumerate}
            \item[a1)] If $y$ and $z$ are not in the same class then they determine a new non-special element from the third class through an orange triple. Together with the element from $A' \cup B' \cup C'$ one of these elements will form a triple which gives another new element from $A \cup B \cup C$. Either the closure meets $U$ in more than 3 elements or trace back to case 1.b).
            \item[a2)] WLOG we can assume that $z=a_i$ and $y=a_j$. These two elements determine some $\beta_k$ due to a black triple. Now if $x=\beta_l$ then from $\beta_k$ and $\beta_l$ we can get a $c_m$ due to a brown triple and then apply 1.a1). If $x=\gamma_l$ then at least one of the pairs $\gamma_l,a_i$ or $\gamma_l,a_j$ can determine a new element $\gamma_m$ via a brown triple and we get a situation like in case 2. If $x=\alpha_l$ then $\alpha_l$, $\beta_k$ determine some $\gamma_m$ through an orange triple and we get back the previous case.
        \end{enumerate}
        \item[b)] $|\{y,z\} \cap \{a,b,c\}|\ne0$:

        WLOG suppose that $z=a$. Now $a$ together with the element from $A' \cup B' \cup C'$ will determine a new element from $U$ hence we trace back to case 1.b) or 1.c).
    \end{enumerate}
    \item $|\{x,y,z\} \cap (A \cup B \cup C)|=1$:
    \begin{enumerate}
        \item[a)] If the two elements from $A' \cup B' \cup C'$ are in different classes then via an orange triple they determine a new element from the third class and together with the element from $A \cup B \cup C$ they can determine at least one new element which is either in $A' \cup B' \cup C'$ and we are done or from $A \cup B \cup C$ thus trace back to case 3.
        \item[b)] If the two elements from $A' \cup B' \cup C'$ are in the same class then they determine a new element from $U$ and we trace back to case 3.
    \end{enumerate}
\end{enumerate}
\end{document}